\documentclass[12pt]{article}

\usepackage{amssymb,amsmath}
\usepackage{amsbsy}
\usepackage{amsthm}
\usepackage{mathrsfs}
\usepackage{verbatim}
\usepackage[colorlinks]{hyperref}
\usepackage{fullpage}
\usepackage{mathdots}
\usepackage{graphicx,subfigure}
\usepackage[english]{babel}
\usepackage{dsfont}
\usepackage[utf8]{inputenc}
\usepackage{tikz}

\usetikzlibrary{backgrounds,fit, matrix}
\usetikzlibrary{positioning}
\usetikzlibrary{calc,through,chains}
\usetikzlibrary{arrows,shapes,snakes,automata, petri}

\newtheorem{Theorem}[equation]{Theorem}
\newtheorem{Corollary}[equation]{Corollary}
\newtheorem{Lemma}[equation]{Lemma}
\newtheorem{Proposition}[equation]{Proposition}

\theoremstyle{definition}
\newtheorem{Definition}[equation]{Definition}
\newtheorem{Example}[equation]{Example}

\theoremstyle{remark}

\newtheorem{Remark}[equation]{Remark}

\numberwithin{equation}{section}
\numberwithin{figure}{section}

\newcommand{\C}{{\mathbb C}}
\newcommand{\Z}{{\mathbb Z}}
\newcommand{\Q}{{\mathbb Q}}
\newcommand{\R}{{\mathbb R}}

\newcommand{\N}{{\mathbb N}}

\newcommand{\mc}[1]{\mathcal{#1}}

\newcommand{\mt}[1]{\text{#1}}
\newcommand{\md}[1]{\mathds{#1}}

\begin{document}

\title{A representation on the labeled rooted forests}
\author{Mahir Bilen Can}

\normalsize

\date{April 4, 2017}
\maketitle

\begin{abstract}
We consider conjugation action of symmetric group on the semigroup of all partial functions and 
develop a machinery to investigate character formulas and multiplicities. In particular, we determine
nilpotent matrices whose orbit under symmetric group afford the sign representation. 
Applications to rook theory are offered.

\vspace{.4cm}
\end{abstract}

\section{Introduction}\label{S:Introduction}

Our goal in this paper is to contribute to the general field of combinatorial representation theory 
by using some ideas from semigroup theory, rook theory, as well as graph theory. 
Classical rook theory is concerned with the enumerative properties of file and rook numbers in 
relation with other objects of mathematics~\cite{BCHR}. In particular, enumeration of functions satisfying 
various constraints falls into the scope of rook theory. Here, we focus on {\em partial functions},
also known as partial transformations, on $[n]:=\{1,\dots, n\}$ with the property that the associated 
graph of the function is a labeled rooted forest.

There is an obvious associative product on the set of all partial transformations on $[n]$; the composition
$f\circ g$ of two partial transformations $f$ and $g$ is defined when the domain of $f$ intersects the 
range of $g$. The underlying semigroup, denoted by $\mc{P}_n$ is called the partial transformation 
semigroup~\cite{GanyushkinMazorchuk}. Of course, the identity map on $[n]$ is a partial transformation 
whence $\mc{P}_n$ is a monoid. Moreover, the subset $\mc{R}_n \subset \mc{P}_n$ consisting of 
injective partial transformations forms a submonoid of $\mc{P}_n$. $\mc{R}_n$ is known as (among combinatorialists)
the rook monoid since its elements have interpretations as non-attacking rook placements on the ``chessboard''  
$[n]\times [n]$ (see~\cite{GR86, KaplanskyRiordan}). It has a central place in the structure theory of reductive 
algebraic monoids. See~\cite{Putcha88,Renner05}.


The symmetric group $S_n$ is the group of invertible elements in both of the monoids $\mc{P}_n$ and $\mc{R}_n$.
In this work, we compute the decompositions of certain representation $S_n$ on ``nilpotent partial transformations.''
Since there is no obvious 0 element in $\mc{P}_n$, we explain this using a larger monoid. 
Let $\mc{T}_n$ denote the {\em full transformation semigroup} which consists of all maps from $\{0,1,\dots, n\}$
into $\{0,1,\dots,n\}$. Clearly, the partial transformation semigroup $\mc{P}_n$ is canonically isomorphic to the 
subsemigroup $\mc{P}_n^* \subset \mc{T}_n$ consisting of elements $\alpha : \{0,1,\dots, n\}\rightarrow \{0,1,\dots, n\}$
such that $\alpha(0)=0$. An element $g\in \mc{T}_n$ is called nilpotent if there exists a sufficiently large $k\in \N$ such that 
for any $x\in \{0,1\dots, n\}$, $g^k (x) = g ( g ( \cdots (x) \cdots )) = 0$. An element $f$ of $\mc{P}$ is called nilpotent 
if, under the canonical identification $f \mapsto g=g_f$ of $\mc{P}$ with $\mc{P}^*$, the corresponding (full) transformation 
$g$ is nilpotent.  

The sets of nilpotent elements of $\mc{P}_n$ and $\mc{R}_n$ are denoted by $\mt{Nil}(\mc{P}_n)$ and $\mt{Nil}(\mc{R}_n)$,
respectively.  There is a beautiful way of representing, in terms of graphs, of the elements of these sets of nilpotent 
transformations. To build up to it, we first mention some useful alternative ways of representing elements of $\mc{P}_n$.

Recall that a partial transformation is a function $f :A \rightarrow [n]$ that is defined on a subset $A$ of $[n]$. 
We write the data of $f$ as a sequence $f=[f_1,f_2,\dots, f_n]$, where $f_i = f(i)$ if $i\in A$, and $f_i = 0$ otherwise. 
Equivalently, $f$ is given by the matrix $f=(f_{i,j})_{i,j=1}^n$ defined by 
$$
f_{i,j} = 
\begin{cases}
1 & \text{ if } i\in A, \\
0 & \text{ otherwise.}
\end{cases}
$$
Conveniently, in the matrix notation, the composition operation on partial functions transfers to the matrix multiplication.

Finally, a more combinatorial way of representing $f\in \mc{P}_n$ is described as follows. 
Starting with $n$ labeled vertices (labeled by the elements of $[n]$), if $f(i)=j$, then we connect the vertex with label $i$ by an
outgoing directed edge to the vertex with label $j$. The resulting graph is called the digraph of the partial transformation $f$.  
In Figure~\ref{F:a file placement} we depict three different representations, including the digraph, of the partial transformation 
$f=[3,3,5,0,5,0,1]$.
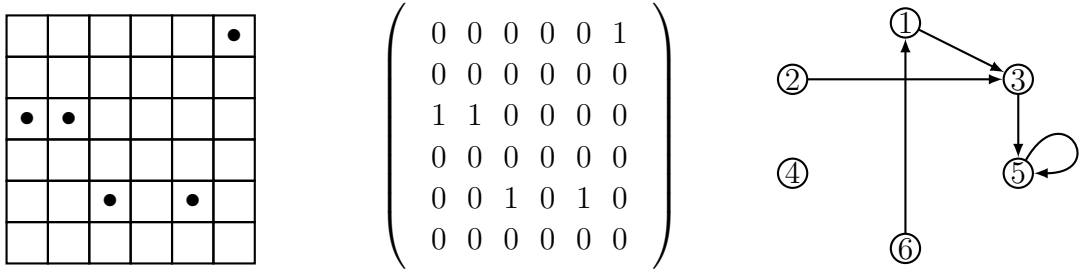
\begin{figure}[htp]
\centering
\begin{tikzpicture}[cap=round,>=latex]

\begin{scope}[xshift=-7.5cm, yshift=-2.25cm,scale=.55]
\node at (1.5,4.5) {$\bullet$};
\node at (2.5,4.5) {$\bullet$};
\node at (3.5,2.5) {$\bullet$};
\node at (5.5,2.5) {$\bullet$};
\node at (6.5,6.5) {$\bullet$};
\foreach \x in {1,...,6}
{\draw[ thick] (\x,6) rectangle +(1,1);}
\foreach \x in {1,...,6}
{\draw[ thick] (\x,5) rectangle +(1,1);}
\foreach \x in {1,...,6}
{\draw[ thick] (\x,4) rectangle +(1,1);}
\foreach \x in {1,...,6}
{\draw[ thick] (\x,3) rectangle +(1,1);}
\foreach \x in {1,...,6}
{\draw[ thick] (\x,2) rectangle +(1,1);}
\foreach \x in {1,...,6}
{\draw[ thick] (\x,1) rectangle +(1,1);}
\end{scope}

\begin{scope}
\matrix [matrix of math nodes,left delimiter=(, right delimiter=) ]
  {
\node {0}; & \node{0}; & \node {0}; &\node {0}; & \node{0}; & \node {1}; \\
\node {0}; & \node{0}; & \node {0}; &\node {0}; & \node{0}; & \node {0}; \\
\node {1}; & \node{1}; & \node {0}; &\node {0}; & \node{0}; & \node {0}; \\
\node {0}; & \node{0}; & \node {0}; &\node {0}; & \node{0}; & \node {0}; \\
\node {0}; & \node{0}; & \node {1}; &\node {0}; & \node{1}; & \node {0}; \\
\node {0}; & \node{0}; & \node {0}; &\node {0}; & \node{0}; & \node {0}; \\
  };
\end{scope}

\begin{scope}[xshift=5cm,-> , thick,loop/.style={min distance=10mm,in=0,out=60,looseness=10}]
 \node[inner sep =0.5, circle,draw] (1) at (0,1.5) {1};
 \node[inner sep =0.5, circle,draw] (2) at (-1.5,.75) {2};
 \node[inner sep =0.5, circle,draw] (3) at (1.5,.75) {3};
 \node[inner sep =0.5, circle,draw] (4) at (-1.5,-.5) {4};
 \node[inner sep =0.5, circle,draw] (5) at (1.5,-.5) {5};
 \node[inner sep =0.5, circle,draw] (6) at (0,-1.5) {6};
\draw[->,thick] (1) to (3);
\draw[->,thick] (2) to (3);
\draw[->,thick] (3) to (5);
\draw[->,thick] (6) to (1);
\draw[->,thick,loop below] (5) to (5);
\end{scope}

\end{tikzpicture}
\caption{Different presentations of the partial transformation $f=[3,3,5,0,5,0,1]$.}
\label{F:a file placement}
\end{figure}

The use of digraphs in rook theory goes back to Gessel's creative work~\cite{Gessel}. This approach is taken much 
afar by Haglund in~\cite{Haglund96} and Butler in~\cite{Butler}. As far as we are aware of, the nilpotent rook placements 
in combinatorics made its first appearance in Stembridge and Stanley's influential work~\cite{StembridgeStanley} on 
the immanants of Jacobi-Trudi matrices, where, essentially, the authors consider only those rook placements fitting 
into a staircase shape board. However, Stembridge and Stanley do not pursue the representation theoretic properties 
as we do here.

Following the terminology of~\cite{BCHR}, we call $f\in \mc{P}_n$ a $k$-file placement if the rank of its matrix 
representation is $k$, and similarly, we call $f\in \mc{R}_n$ a $k$-rook placement if the rank of its matrix 
representation is $k$. The basic observation that our paper builds on is that the symmetric group acts on nilpotent $k$-file 
placements as well as on the nilpotent $k$-rook placements. In fact, there is a more general statement from semigroup theory: 
the unit group of a monoid (with 0) acts on the set of nilpotent elements of the semigroup. 
Here we focus on the partial transformation monoid and the rook monoid. The main reason for confining ourselves 
only to these two special semigroups is twofold. First of all, the resulting objects from our investigations, namely, 
the labeled rooted forests, has significance not only in modern algebraic combinatorics 
(\cite{GarsiaHaiman96,HHLR,HaglundBook,Haiman03}, see~\cite{AkerCan12}, also) but also in the theory of 
classical transformation semigroups (see \cite{GanyushkinMazorchuk}). Secondly, much studied but still mysterious 
plethysm operation (\cite{Littlewood44,Howe88}) from representation theory has a very concrete combinatorial 
appearance in our work. Furthermore, rooted trees and forests play a very important role for computer science 
(see~\cite{Knuth3}). There are plenty of other reasons to focus on these objects for statistical and probabilistic purposes.

Now we are ready to give a brief overview of our paper and state our main results. In Section~\ref{S:Preliminaries}, we 
introduce the necessary notation and state some of the results that we use in the sequel. The purpose of Section~\ref{S:recounting} 
is to give a count of the number of file placements (partial transformations) according to the sizes of their domains. It turns out this 
count is the same as that of the ``labeled rooted forests.'' We prove in our Theorem~\ref{T:number of k files} that the number of 
nilpotent $k$-file placements is equal to ${n-1 \choose k} n^{k}$. This numerology is the first step towards understanding the 
$S_n$-module structure on the set of all nilpotent file placements. Towards this goal, we devote whole Section~\ref{S:case study}
to study $n=3$ case.

At the beginning of Section~\ref{S:induced representation} we show that the conjugation action on a nilpotent file placement 
does not alter the underlying unlabeled rooted forest. Moreover, we observe that the action is transitive on the labels. 
Let $\sigma$ be a labeled rooted forest. We denote the resulting representation, that is to say, the orbit of $\sigma$ under 
the conjugation action, by $o(\sigma)$ and call it the odun of $\sigma$. We observe in Theorem~\ref{T:add a vertex} that the 
representation $o=o(\sigma)$ is related by a simple operation to $o(\tau)$ if $\tau$ is the labeled rooted forest that is obtained 
from $\sigma$ by removing the root. As a simple consequence of this fact, we obtain our first recursive relation among the 
characters of oduns of forests. Another important observation we make in the same sections is that the adding of $k$ new isolated 
vertices to a given rooted forest corresponds to tensoring the original representation by the $k$-dimensional standard representation 
of $S_k$. This is our Theorem~\ref{T:add a regular representation}.
There is even more general statement that we record in Remark~\ref{R:more general}: 
If the rooted forest $\sigma$ is written as a disjoint union $\tau \cup \nu$ of two other rooted forests which do not have any identical  
rooted subtrees, then $o(\sigma)=o(\tau) \otimes o(\nu)$.

In Section~\ref{S:Plethysm} we prove our master plethysm result, Theorem~\ref{T:master plethysm theorem}, 
which states that if a rooted forest $\tau$ (on $mk$ vertices) is comprised of $k$ copies of the same rooted tree $\sigma$ 
(on $m$ vertices), then the odun of $\tau$ is given by the compositional product of the standard $k$-dimensional 
representation of $S_k$ with $o(\sigma)$. Combined with Remark~\ref{R:more general} this result gives us a satisfactorily 
complete description of the rooted forest representations. 

The main purpose of Section~\ref{S:Sign} is to determine when sign representation occurs in a given odun. 
The surprising combinatorial result of this section states that the sign representation occurs in $o(\sigma)$ 
if and only if $\sigma$ is ``blossoming.'' A rooted forest is called {\em blossoming} if it has no rooted subtree having (at least)
two identical maximal terminal branches of odd length emanating from the same vertex. 
By counting blossoming forests, we show that the total occurrence of sign representation in all rooted forest representations 
(on $n$ vertices) is equal to $2^{n-2}$.

The labeled rooted forest associated with a non-attacking rook placement has a distinguishing feature; it is a union of chains.
From representation theory point of view the ordering of the chains does not matter, therefore, there is a correspondence 
between the number of rooted forest representations (of non-attacking rooks) and partitions. In Section~\ref{S:Rooks} we 
make this precise. These results are simple applications of the previous sections.

In our ``Final Remarks'' section we give a formula for the dimension of a forest representation and compare our result with 
Knuth's hook-length formula. Finally, we close our paper in Section~\ref{S:Tables} by presenting tables of irreducible constituents 
of the nilpotent $k$ file placements.

\vspace{.5cm}

\textbf{Acknowledgements.}
We thank Michael Joyce, Brian Miceli, Jeff Remmel, and Lex Renner.

\section{Preliminaries}\label{S:Preliminaries}

\subsection{Terminology of forests}

A {\em rooted tree} $\sigma$ is a finite collection of vertices such that there exists a designated vertex, called the root
(or the {\em ancestor}), and the remaining vertices are partitioned into a finite set of disjoint non-empty subsets 
$\sigma_1,\dots, \sigma_m$, each of which is a tree itself. We depict a tree by putting its root at the top so that the 
following terminology is logical: if a vertex $a$ is connected by an edge to another vertex $b$ that is directly above $a$, 
then $a$ is called a {\em child} of $b$. Any collection of rooted trees is called a rooted forest. In particular, a tree is a forest. 
The elements of the set $[n]=\{1,\dots, n\}$ are often referred to as {\em labels}.
Cayley's theorem states that there are $n^{n-1}$ labeled rooted trees on $n$ vertices. 
Since adding a new vertex to a forest on $n$ vertices results in a tree on $n+1$ vertices (by connecting
the roots to the new vertex), and vice versa, 
Cayley's theorem is equivalent to the statement that there are $(n+1)^{n-1}$ labeled rooted forests on $n$ vertices.

\subsection{Basic character theory}

It is well-known that the irreducible representations of $S_n$ are indexed by partitions of $n$.
If $\lambda$ is a partition of $n$ (so we write $\lambda \vdash n$), 
then the corresponding irreducible representation is denoted by $V_\lambda$.

Let $R^0$ denote the ring of integers $\Z$ and let $R^n$, $n=1,2,\dots$ denote the $\Z$-module spanned by irreducible 
characters of $S_n$. We set $R=\bigoplus_{n\geq 0} R^n$. In a similar fashion, let $\Lambda$ denote the direct sum 
$\bigoplus_{n\geq 0} \Lambda^n$, where $\Lambda^n$ is the $\Z$-module spanned by homogenous symmetric functions 
of degree $n$, and $\Lambda^0=\Z$. Both of these $\Z$-modules are in fact $\Z$-algebras, and the Frobenius characteristic map 
$$
R \ni \chi \mapsto \mt{ch}(\chi) = \sum_{\rho \vdash n} z_\rho^{-1} \chi_\rho p_\rho \in \Lambda
$$
is a $\Z$-algebra isomorphism. Here, $z_\rho$ denotes the quantity $\prod_{i\geq 1} i^{m_i} m_i !$, where $m_i$ is the number of 
occurrence of $i$ as a part of $\rho$, $\chi_\rho$ is the value of the character $\chi$ on the conjugacy class indexed by 
the partition $\rho=(\rho_1,\dots, \rho_r)$, and $p_\rho$ is the power sums symmetric function $p_\rho = p_{\rho_1} \cdots p_{\rho_r}$.
Under Frobenius characteristic map, the irreducible character $\chi^\lambda$ of the representation $V_\lambda$ 
is mapped to the Schur function $s_\lambda:= \sum_{\rho \vdash n} z_\rho^{-1} \chi^\lambda_\rho p_\rho$. (This can be taken as 
the definition of a Schur function.)
The monomial symmetric function associated with partition $\lambda=(\lambda_1,\dots, \lambda_l)$ is defined as the sum of all monomials 
of the form $x_1^{\beta_1}\cdots x_l^{\beta_l}$, where $(\beta_1,\dots, \beta_l)$ ranges over all distinct permutations of $(\lambda_1,\dots, \lambda_l)$.
Any of the sets $\{s_\lambda\}_{\lambda \vdash n}, \{m_\lambda\}_{\lambda \vdash n}$, and
$\{p_\lambda\}_{\lambda \vdash n}$ forms a $\Q$-vector space basis for the vector space $\Lambda^n \otimes_\Z \Q$. 
Kostka numbers $K_{\lambda \mu}$ are defined as the coefficients in the expansion $s_\lambda = \sum_{\mu \vdash n} K_{\lambda \mu} m_\mu$.

\subsection{Symmetric functions and plethysm}

The plethysm of the Schur functions $s_\lambda \circ s_\mu$ is the symmetric function obtained from $s_\lambda$ by 
substituting the monomials of $s_\mu$ for the variables of $s_\lambda$. To spell this out more precisely we follow~\cite{LoehrRemmel}. 
The plethysm operator on symmetric functions is the unique map $\circ:\ \Lambda \times \Lambda \rightarrow \Lambda$ satisfying 
the following three axioms: 
\begin{enumerate}
\item[P1.] For all $m,n\geq 1$, $p_m \circ p_n = p_{mn}$.
\item[P2.] For all $m\geq 1$, the map $g \mapsto p_m \circ g$, $g\in \Lambda$ defines a $\Q$-algebra 
homomorphism on $\Lambda$. 
\item[P3.] For all $g\in \Lambda$, the map $h \mapsto h \circ g$, $h\in \Lambda$ defines a $\Q$-algebra 
homomorphism on $\Lambda$. 
\end{enumerate}

In general, computing the plethysm of two arbitrary symmetric functions is not easy.
Fortunately, there are some useful formulas involving Schur functions:
\begin{align}\label{A:plethysm formula 1}
s_\lambda \circ (g+h) &= \sum_{\mu,\nu} c_{\mu,\nu}^\lambda (s_\mu \circ g) (s_\nu \circ h),
\end{align}
and 
\begin{align}\label{A:plethysm formula 2}
s_\lambda \circ (gh) &= \sum_{\mu,\nu} \gamma_{\mu,\nu}^\lambda (s_\mu \circ g) (s_\nu \circ h).
\end{align}
Here, $g$ and $h$ are arbitrary symmetric functions, $c_{\mu,\nu}^\lambda$ is a scalar, 
and $\gamma_{\mu,\nu}^\lambda$ is $\frac{1}{n!} \langle \chi^\lambda, \chi^\mu \chi^\nu \rangle$.
In~(\ref{A:plethysm formula 1}) the summation is over all pairs of partitions $\mu,\nu \subset \lambda$,
and the summation in~(\ref{A:plethysm formula 2}) is over all pairs of partitions $\mu,\nu$ such that 
$|\mu| = |\nu | = |\lambda|$.
In the special case when $\lambda = (n)$, or $(1^n)$ we have
\begin{align}
s_{(n)} \circ (gh) &= \sum_{\lambda \vdash n} (s_\lambda \circ g) (s_\lambda \circ h), \label{A:complete}\\
s_{(1^n)} \circ (gh) &= \sum_{\lambda \vdash n} (s_\lambda \circ g) (s_{\lambda'} \circ h), \label{A:elementary}
\end{align}
where $\lambda'$ denotes the conjugate of $\lambda$.

In a similar vein, if $\rho$ denotes a partition, then the coefficient of $s_\mu$ in $p_\rho \circ h_n$ is given by $K^{(\rho)}_{\mu,n\rho}$,
{\em the generalized Kostka numbers}. Since we need this quantity in one of our calculations, we define it. 
A generalized tableau of type $\rho$ shape $\mu$ and weight $n\rho = (n\rho_1,n\rho_2,\dots, n\rho_m)$
is a sequence $T=(\nu^{(0)},\nu^{(1)},\dots, \nu^{(m)})$ of partitions satisfying the following conditions 
\begin{enumerate}
\item $0= \nu^{(0)} \subset \nu^{(1)} \subset \cdots \subset \nu^{(m)}=\mu$;
\item $|\nu^{(j)} - \nu^{(j-1)}| = n\rho_j$ for $1\leq j \leq m$;
\item $\nu^{(j)} \approx_{\rho_j} \nu^{(j-1)}$ for $1\leq j \leq m$. (See~\cite{Mac} Chapter I, $\S 5$, Example 24 for definition of $\approx_{\rho_j}$.)
\end{enumerate}
For such tableau, define $\sigma(T) := \prod_{j=1}^m \sigma_{\rho_j} (\nu^{(j)}/\nu^{(j-1)})$ which is $\pm 1$. 
(Once again, see~\cite{Mac} Chapter I, $\S 5$, Example 24 for definition of $\sigma_{\rho_j}(\cdot )$). Finally, we define $K^{(\rho)}_{\mu,n\rho}$ as 
the sum of $\sigma(T)$'s
$$
K^{(\rho)}_{\mu,n\rho} := \sum_{T} \sigma(T),
$$
where the sum ranges over all generalized tableau $T$ of type $\rho$, shape $\mu$, and of weight $n\rho$.

\section{Re-counting nilpotent rooks}\label{S:recounting}

Recall that a nilpotent file placement is the one with no cycles in its associated directed labeled graph.
We have a simple lemma re-interpreting this definition using matrices.

\begin{Lemma}
A file placement is nilpotent if and only if its associated matrix (as defined in Section~\ref{S:Introduction}) 
is nilpotent. 
\end{Lemma}

\begin{proof}
Let $f :A\rightarrow [n]$ denote the partial transformation representing a file placement on $[n]\times [n]$. 
If the graph of $f$ has a non-trivial cycle, then there exists a sequence $i_1,\dots, i_r$ numbers from 
the domain $A$ of $f$ such that $f(i_1) = i_2,f(i_2) =i_3,\dots, f(i_r ) = i_1$. 
It follows that for any $m>0$, $f^m (i_1) = f^{m\mod r} ( i_1) \in \{ i_1,\dots, i_r\}$, hence 
no power of $f$ can be zero. 
Conversely, if $f$ is nilpotent, then for any $a\in A$, some power of $f$ vanishes on $a$.
Therefore, $a,f(a),f^2(a),\dots$ does not return to $a$ to become a cycle.
\end{proof}

\begin{Remark}\label{L:LaradjiUmar}
The proof of the above lemma implies that a file placement (hence, a non-attacking rook placement) 
$f : A\rightarrow [n]$ is nilpotent if and only if there does {\em not} exist a subset $D\subset A$ such that $f (D) = D$. 
This observation is recorded in~\cite{LaradjiUmar04}.
\end{Remark}

\begin{Theorem}\label{T:number of k files}
The number of nilpotent $k$-file placements is equal to ${n-1 \choose k} n^{k}$.
\end{Theorem}

\begin{proof}
It is a well known variation of the Cayley's theorem that the number of labeled forests on $n$ vertices with $k$ roots
is equal to ${n-1 \choose k-1} n^{n-k}$. See~\cite{Comtet}, Theorem D, pg 70.
Since the labeled directed graph of a nilpotent partial transformation has no cycles, it is a disjoint union of 
trees and the total number of vertices is $n$. Therefore, it remains to show that the labeled 
forest of a $k$-file placement $f$ has exactly $n-k$ connected components. 
We prove this by induction on $k$.

If $k=1$, then the forest of $f$ has one component on two vertices and $n-2$ singletons. 
Therefore,  the base case is clear. Now we assume that our claim is true for $k-1$ and prove it 
for $k$-file placements. Let $j \in A$ be a number that is not contained in the image of $f$
whose existence is guaranteed by Lemma~\ref{L:LaradjiUmar}.
Define $\tilde{f} : A - \{j\} \rightarrow [n]$ by setting $\tilde{f}(i) = f(i)$ for $i\in A - \{j\}$. 
Therefore, $\tilde{f}$ is a $(k-1)$-file placement agreeing with $f$ at all places except at 
$\{j\}$, where it is undefined.
By our induction hypothesis, the forest of $\tilde{f}$ has exactly $n-(k-1)$ connected components. 
Observe that the forest of $f$ differs from that of $\tilde{f}$ by exactly one directed edge from $j$ to $f(j)$. 
Since $\{ j \}$ is a connected component of $\tilde{f}$, the number of connected components of $f$ 
is one less than that of $\tilde{f}$, hence the proof is complete.

\end{proof}

\begin{Remark}\label{R:LaradjiUmar}
Theorem~\ref{T:number of k files} gives the number of partial transformations on $[n]$ having exactly $k$ elements in their domain. 
There is a similar count for the partial transformations in ~\cite{LaradjiUmar04}. 
For completeness of the section let us briefly present this:
Let $N(J_r)$ denote the set of partial transformations $f: A \rightarrow [n]$ with $|f(A)| = r$. 
In their Theorem 3, Laradji and Umar compute that 
$$
|N(J_r)|= {n \choose r} S(n, r+1) r!,
$$
where $S(n,r+1)$ is the Stirling number of the second kind, namely the number of set partitions of $[n]$ into $r+1$ non-empty blocks. 
\end{Remark}

$N_{k,n}$ denote the number of nilpotent $k$-file placements in $[n]\times [n]$. 
Let us say a few words about the exponential generating series of $N_{k,n}$. Define $E_x(y)$ by 
\begin{align}\label{A:egf of nilpotent files}
E_x(y):=\sum_{n\geq 1} \sum_{k\geq 0}^{n-1} N_{k,n} x^k\frac{y^n}{n!} =  \sum_{n \geq 1} (1+nx)^{n-1} \frac{y^n}{n!}.
\end{align}
For the following identities, see~\cite[Chapter 5]{GKP}.
\begin{enumerate}
\item $E^{-x} \ln E = y$,
\item $E^\alpha = \sum_{n\geq 0} \alpha (\alpha + nx)^{n-1} \frac{y^n}{n!}$ for all $\alpha \in \R$,
\item $\frac{E^\alpha}{1-xyE^x} = \sum_{n\geq 0}  (\alpha + nx)^{n} \frac{y^n}{n!}$ for all $\alpha \in \R$.
\end{enumerate}
Let $t_{n+1}$ denote the number of rooted trees on $n+1$ vertices. 
Manipulation of the generating functions lead to the following non-trivial recurrence for $t_n$'s:
\begin{align}\label{A:non-trivial}
t_{n+1} = \frac{1}{n} \sum_{k=1}^n \left( \sum_{d \mid k } d t_d \right) t_{n-k+1} \qquad (t_1=1).
\end{align}
This equation indicates that the number of nilpotent file placements is not as easily expressible as one wishes.

\section{A case study}\label{S:case study}

We start with fixing our notation. The set of all nilpotent $k$-file placements on $[n]\times [n]$ board is denoted by $\mc{C}_{k,n}$.
In other words, $\mc{C}_{k,n}= \mt{Nil}(\mc{P}_n) \cap \mc{P}_k ([n]\times [n])$. Let $\mc{C}_n$ denote the union  
\begin{align*}
\mc{C}_n = \bigcup_{k=0}^{n-1} \mc{C}_{k,n}.
\end{align*}

Obviously, there is a single nilpotent partial transformation on $\{1\}$.
For $n=2$, the nilpotent partial transformations are 
\begin{align*}
\mc{C}_{0,2} = \{[0,0]\},\qquad
\mc{C}_{1,2} = \{ [0, 1], [2,0] \},
\end{align*}
and for $n=3$ we have
\begin{align*}
\mc{C}_{0,3} &= \{[0,0,0]\},\\
\mc{C}_{1,3} &= \{ [0, 1, 0], [0, 3, 0], [0, 0, 1], [0, 0, 2], [2, 0, 0], [3, 0, 0] \},\\ 
\mc{C}_{2,3} &= \{ [0, 1, 2], [0, 1, 1], [0, 3, 1], [2, 0, 1], [2, 0, 2], [2, 3, 0], [3, 0, 2], [3, 1, 0], [3, 3, 0]\}. 
\end{align*}
Symmetric group $S_3$ acts on each $C_{i,3}$, $i=0,1,2$ by conjugation. 
The table of corresponding character values are easy to determine by counting fixed points of the action.
\begin{table}[htp]
\begin{center}
\begin{tabular}{c | c c c }
$g$   &  $\mc{C}_{0,3}$ &  $\mc{C}_{1,3}$ & $\mc{C}_{2,3}$  \\
\hline
(1)(2)(3)    & 1 & 6 & 9   \\
(12)(3)      & 1 &  0 & 1     \\
(123)        & 1 &  0 &  0\\
\end{tabular}
\end{center}
\caption{Character values of the $S_3$ representation on nilpotent 3-file placements}
\label{T:1}
\end{table}
(The first column in Table~\ref{T:1} is the list of representatives for each conjugacy class in $S_3$.)
In the next table we have the character values of all irreducible representations of $S_3$. 
In the last column, we have listed the sizes of the corresponding conjugacy classes:
\begin{table}[htp]
\begin{center}
\begin{tabular}{c | c c c | c}
$g\in S_3$  &  $V_{(3)}$ &  $V_{(2,1)}$ & $V_{(1,1,1)}$ & $c(g)$ \\
\hline
 (1)(2)(3)    & 1 & 2 & 1  & 1 \\
 (12)(3)      & 1 &  0 & -1  & 3   \\
 (123)        & 1 &  -1 &  1 & 2\\
\end{tabular}
\end{center}
\caption{Irreducible character values of $S_3$}
\label{T:2}
\end{table}
Now, from Tables~\ref{T:1} and~\ref{T:2} we see that 
\begin{itemize}
\item $\mc{C}_{0,3}$ is the trivial representation $V_{(3)}$. 
\item $\mc{C}_{1,3}$ is equal to $V_{(3)} \oplus V_{(2,1)}^2 \oplus V_{(1,1,1)}$. 
\item $\mc{C}_{2,3} = V_{(3)}^2 \oplus V_{(2,1)}^3 \oplus V_{(1,1,1)}$. 
\end{itemize}

Going back to cases $n=1$ and $n=2$, we compute also that, as a representation of $S_1$, 
$\mc{C}_{0,1}$ is the unique irreducible (trivial) representation $V_{(1)}$ of $S_1$.
Similarly, $\mc{C}_{0,2} = V_{(2)}$, and $\mc{C}_{1,2} = V_{(2)} \oplus V_{(1,1)}$.
We listed the decomposition tables for $n=4,5$ and $n=6$ at the end of the paper.

\section{Induced representations and labeled rooted forests}\label{S:induced representation}

It follows from the proof of Theorem~\ref{T:number of k files} that there is a correspondence 
between nilpotent $k$-file placements and labeled rooted forests on $n-k$ components.   
If $f$ is a $k$-placement, we denote by $\sigma=\sigma_f$ the corresponding labeled rooted forest.
Let us denote by $\chi^{n}:= \chi^{\mc{C}_{n}}$ the character of the conjugation action on nilpotent file placements
and denote by $\chi^{k,n}:= \chi^{\mc{C}_{k,n}}$ the character of the conjugation action on nilpotent $k$-file placements.
Since $S_n$ action does not change the number of rooks, we have $\chi^n= \sum_{k=0}^{n-1} \chi^{k,n}$. 
The module $\mc{C}_{0,n} = V_{(n)}$ is the trivial representation of $S_n$, so 
we are going to focus on the cases where $k\geq 1$.

\begin{Lemma}\label{L:odun}
For $i=1,\dots, n-1$, let $(i,i+1) \in S_n$ denote the simple transposition that interchanges $i$ and $i+1$. 
If $\sigma=\sigma_f$ is a labeled rooted tree corresponding to a nilpotent $k$-file placement $f$, 
then the underlying unlabeled rooted tree of $(i,i+1)\cdot \sigma$ is equal to that of $\sigma$. 
Moreover, if $o$ is a fixed unlabeled rooted tree, then $S_n$ acts transitively on the set of elements $\sigma \in \mc{C}_{k,n}$
whose underlying tree is equal to $o$.
\end{Lemma}

\begin{proof}
Suppose $[f_1,f_2,\dots, f_n]$ is the one line notation for $f$. The conjugation action of $s_i$ on $f$ has the following effect: 
1) the entries $f_i$ and $f_{i+1}$ are interchanged,
2) if $f_l = i$ for some $1\leq l \leq n$, then $f_l$ is replaced by $i+1$. Similarly, if $f_{l'} = i+1$ for some $1\leq l' \leq n$, 
then $f_{l'}$ is replaced by $i$. These operations do not change the underlying graph structure, they act as permutations 
on labels only. Hence, $u((i,i+1)\cdot \sigma)=u(\sigma)$. 

To prove the last statement we fix a labeled rooted tree $\sigma$. It is enough to show the existence of a permutation
which interchanges two chosen labels $i$ and $k$ on $\sigma$ without changing any other labels. Looking at the one-line 
notation for $f$, we see that the action of transposition $(i,k)$ gives the desired result.
\end{proof}

{\it Caution:}
Recall our terminology from the introductory section; the odun $o(\sigma)$ of $\sigma$ is the $S_n$-representation on 
the orbit $S_n\cdot \sigma$. By Lemma~\ref{L:odun}, we see that the odun is completely determined by the underlying 
(unlabeled) rooted tree. Therefore, if there is no danger of confusion, we use the word ``odun'' for the underlying unlabeled 
structure as well.

\begin{Corollary}
The multiplicity of the trivial character in $\chi^n$ is the number of rooted forests on $n$ vertices. 
Equivalently, $\langle \chi^{(n)}, \chi^n \rangle = $ number of rooted trees on $n+1$ vertices. 
\end{Corollary}
\begin{proof}
This is a standard fact: The multiplicity of the trivial representation in any permutation representation is equal to the number 
of orbits of the action. By Lemma~\ref{L:odun}, this number is equal to the number of oduns (unlabeled rooted trees).
\end{proof}

\begin{Example}\label{E:n=3}

Let $\sigma$ be a labeled rooted forest, and let $o=o(\sigma)$ denote its odun. We denote the corresponding 
character by $\chi^{o(\sigma)}$. Here, we produce three examples of forest representations, decomposed into 
irreducibles that we use in the sequel.

\begin{figure}[htp]
\centering
\begin{tikzpicture}[scale=.5]
\begin{scope}[yshift=3.5cm]
\node at (-6,0) {$1.$};
\node at (-2.5,0) {$o(\sigma) = $};
\node at (0,-1) {$\bullet$};
\node at (0,0) {$\bullet$};
\node at (0,1) {$\bullet$};
\node at (9.5,0) {$\implies \qquad \chi^{o(\sigma)} = \chi^{(3)}+2\chi^{(2,1)}+\chi^{(1^3)}$};
\draw[-, thick] (0,-1) to  (0,1);
\end{scope}

\begin{scope}[yshift=0cm]
\node at (-6,0) {$2.$};
\node at (-2.5,0) {$o(\sigma) = $};
\node (1) at (-.5,-.5) {$\bullet$};
\node (2) at (0,.75) {$\bullet$};
\node (3) at (.5,-.5) {$\bullet$};

\node at (8,0) {$\implies \qquad \chi^{o(\sigma)} = \chi^{(2,1)}+\chi^{(1^3)}$};
\draw[-, thick] (-.5,-.5) to  (0,.75);
\draw[-, thick] (.5,-.5) to  (0,.75);
\end{scope}

\begin{scope}[yshift=-3.5cm]
\node at (-6,0) {$3.$};
\node at (-2.5,0) {$o(\sigma) = $};
\node (1) at (-.5,-.5) {$\bullet$};
\node (2) at (-.5,.5) {$\bullet$};
\node (3) at (.5,-.5) {$\bullet$};
\node (4) at (.5,.5) {$\bullet$};

\node at (12.5,0) {$\implies \qquad \chi^{o(\sigma)} = \chi^{(4)}+\chi^{(3,1)}+ 2\chi^{(2,2)}+\chi^{(2,1,1)}+ \chi^{(1^4)}$};
\draw[-, thick] (-.5,-.5) to  (-.5,.5);
\draw[-, thick] (.5,-.5) to  (.5,.5);
\end{scope}
\end{tikzpicture}

\end{figure}

\end{Example}

\begin{Theorem}\label{T:add a vertex}
Let $\sigma \in \mc{C}_{k,n}$ be a labeled rooted forest on $n$ vertices. Let $\chi^{\mathds{1}}$ denote the character 
of the unique (1 dimensional) representation of $S_1$. If $\widetilde{\sigma}$ denotes the labeled rooted tree obtained 
from $\sigma$ by connecting its roots to a new root, which is labeled by $n+1$, then 
$$
\chi^{o({\widetilde{\sigma}})} = \chi^{\mathds{1}}\cdot  \chi^{o(\sigma)}.
$$
\end{Theorem}
\begin{proof}
 
Adding a vertex with label $n+1$ as the unique root transforms $\sigma$ to an element $\widetilde{\sigma}$ in $\mc{C}_{n,n+1}$. 
The $S_{n+1}$-orbit of $\widetilde{\sigma}$ decomposes into exactly $n+1$ orbits. 
If the root of an orbit $A$ has label $i$, then the representation of the Young subgroup $S_1\times S_n$ on $A$ is isomorphic 
to the $S_n$ representation on the set $S_n\cdot \sigma$, which, by definition, is the odun of $\sigma$. 
Note that $S_1\times S_n$ is the stabilizer subgroup in $S_{n+1}$ of the label $i$. Therefore, 
\begin{align}
o(\widetilde{\sigma}) = \oplus_{\pi \in S_{n+1}/S_n} \pi \cdot A.
\end{align}
Notice also that here we are repeating the definition of an induced representation. 
The subgroup $S_n\times S_1$ of $S_{n+1}$ acts on $o(\widetilde{\sigma})$ by fixing the label of the new vertex $n+1$. 
But this representation is isomorphic to $S_n$-module $o(\sigma)$. 
Hence, our proof follows. 
\end{proof}


\begin{Corollary}\label{C:character of n-1}
The character $\chi^{n-1,n}$ of $\mc{C}_{n-1,n}$ is equal to $\chi^{\mathds{1}} (\chi^{0,n-1}+\chi^{1,n-1}+\cdots + \chi^{n-2,n-1})$.
\end{Corollary}
\begin{proof}
This follows from the fact that the elements of $\mc{C}_{n-1,n}$ are obtained from those of $\mc{C}_{n-1}$ by adding 
a single vertex as the new root. 
\end{proof}

The idea of the proof of our next result is identical to that of Theorem~\ref{T:add a vertex}, so we skip it.
\begin{Theorem}\label{T:add a regular representation}
Let $\chi^{(k)}$ denote the character of the standard representation of $S_k$. Suppose that $\sigma$ has $m$ connected 
components. If $\widetilde{\sigma}$ is the labeled rooted forest obtained from $\sigma$ by adding $k$ isolated roots (hence it 
has $m+k$ connected components), then 
$$
\chi^{o({\widetilde{\sigma}})} = \chi^{(k)}\cdot  \chi^{o(\sigma)}.
$$
\end{Theorem}

\begin{Remark}\label{R:more general}
More general than Theorem~\ref{T:add a vertex} with an almost identical proof is the following statement: 
Suppose $\nu$ is a labeled rooted forest of the form $\tau \cup \nu$ (disjoint union), where $\tau$ and $\nu$ are 
labeled rooted trees on $m$ and $k$ vertices, respectively. If $\tau$ and $\nu$ do not have any identical  
connected component, then 
\begin{align}
\chi^{\sigma}= \chi^\sigma \cdot \chi^\tau.
\end{align}
\end{Remark}

Using Theorems~\ref{T:add a regular representation} and~\ref{T:add a vertex} we perform a sample calculation of the 
characters for small $k$. 
\begin{Proposition}\label{P:character of 1 and 2}
For any integer $n$ with $n\geq 3$, we have 
\begin{enumerate}
\item[(i)] $\mc{C}_{1,n} = V_{(n)} \oplus V_{(n-1,1)}^2 \oplus V_{(n-2,2)} \oplus V_{(n-2,1,1)}$.
\item[(ii)] $\mc{C}_{2,n} = (V_{(3)}\oplus V_{(2,1)}^2\oplus V_{(1^3)})\otimes V_{(n-3)} \oplus (V_{(2,1)}^2\oplus V_{(1^3)})\otimes V_{(n-3)}
\oplus (V_{(4)}\oplus V_{(3,1)}\oplus  V_{(2,2)}^2\oplus V_{(2,1,1)}\oplus V_{(1^4)})\otimes V_{(n-4)}$.
\end{enumerate}
\end{Proposition}

\begin{proof}
(i) 
A nilpotent 1 file placement from $\mc{C}_{1,n}$ is a sequence $f=[f_1,\dots, f_n]$ of length $n$ with a unique non-zero entry 
$f_i =j \in [n]$ such that $i\neq j$. As a labeled rooted tree $\sigma=\sigma_f$ is an array of $n$ vertices labeled from 1 to $n$,
and the $i$th vertex is connected to the $j$th by a directed edge. See Figure~\ref{F:C1n}.
\begin{figure}[htp]
\centering
\begin{tikzpicture}[scale=.75,cap=round,>=latex,distance=2cm, thick]

\node[inner sep =0.5, circle,draw] (1) at (-3,0) {1};
\node[inner sep =0.5, circle,draw] (2) at (-2,0) {2};
\node (3) at (-1,0) {$\dots$};
\node[inner sep =0.5, circle,draw] (4) at (0,0) {$i$};
\node (5) at (1,0) {$\dots$};
\node[inner sep =0.5, circle,draw] (6) at (2,0) {$j$};
\node (7) at (3,0) {$\dots$};
\node[inner sep =0.5, circle,draw] (8) at (4,0) {$n$};

\draw[->] (4) to [out = 45, in =135,distance = 1cm, looseness=.5] (6);

\end{tikzpicture}
\caption{The labeled rooted forest form of a typical element of $\mc{C}_{1,n}$.}
\label{F:C1n}
\end{figure}
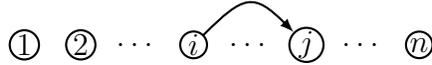
The rest of the proof follows from Theorem~\ref{T:add a regular representation}.

(ii) 
It is easy to verify that the underlying forest of an element $\sigma\in \mc{C}_{2,n}$ is one of the 
three forests which are depicted in Figure~\ref{F:oduns of 2,n}.
\begin{figure}[htp]
\centering
\begin{tikzpicture}[scale=.5]
\begin{scope}[yshift=3.5cm]
\node (0) at (-2,0) {$1.$};
\node (1) at (0,-1) {$\bullet$};
\node (2) at (0,0) {$\bullet$};
\node (3) at (0,1) {$\bullet$};
\node (4) at (.75,0) {$\bullet$};
\node (5) at (2,0) {$\cdots$};
\node (6) at (3,0) {$\bullet$};
\node (7) at (9.5,0) {($n-3$ isolated vertices)};
\draw[-, thick] (0,-1) to  (0,1);
\end{scope}
\begin{scope}[yshift=0cm]
\node (0) at (-2,0) {$2.$};
\node (1) at (0,0) {$\bullet$};
\node (2) at (.5,1) {$\bullet$};
\node (3) at (1,0) {$\bullet$};
\node (4) at (1.75,0) {$\bullet$};
\node (5) at (3,0) {$\cdots$};
\node (6) at (4,0) {$\bullet$};
\node (7) at (9.5,0) {($n-3$ isolated vertices)};
\draw[-, thick] (0,0) to  (0.5,1);
\draw[-, thick] (1,0) to  (0.5,1);
\end{scope}
\begin{scope}[yshift=-3.5cm]
\node (0) at (-2,0) {$3.$};
\node (1) at (0,0) {$\bullet$};
\node (2) at (0,1) {$\bullet$};
\node (3) at (1,0) {$\bullet$};
\node (4) at (1,1) {$\bullet$};
\node (5) at (1.75,0) {$\bullet$};
\node (6) at (3,0) {$\cdots$};
\node (7) at (4,0) {$\bullet$};
\node (8) at (9.5,0) {($n-4$ isolated vertices)};
\draw[-, thick] (0,0) to  (0,1);
\draw[-, thick] (1,0) to  (1,1);
\end{scope}
\end{tikzpicture}
\caption{The oduns of $\mc{C}_{2,n}$.}
\label{F:oduns of 2,n}
\end{figure}
Thus, the proof follows from Theorem~\ref{T:add a regular representation} in the view of Example~\ref{E:n=3}.

\end{proof}

\section{Plethysm and labeled rooted trees.}\label{S:Plethysm}

We start with some preliminary observations. 
First of all, since a labeled rooted tree $\sigma$ is obtained from a collection $\sigma_1,\dots, \sigma_r$ 
of rooted subforests by adding a vertex attached to all of their roots, we know from Theorem~\ref{T:add a vertex}
that $\chi^{o(\sigma)} = \chi^{\md{1}}\prod_{i}^r \chi^{o(\sigma_i)}$. In the light of Remark~\ref{R:more general}, 
we assume that all subforests are distinct in the sense that, if a tree $\tau$ is a connected component in $\sigma_i$,
then (an isomorphic copy of) $\tau$ does not appear in any other subforest $\sigma_j$ as a connected component. 
Therefore, for our purposes it suffices to investigate the representation of $S_{rm}$ 
on a forest which is comprised of $r$ copies of the same rooted tree on $m$ vertices.

\begin{Theorem}\label{T:master plethysm theorem}
Let $\sigma \in \mc{C}_n$ be a labeled rooted forest consisting of $r$ copies of 
the same labeled rooted tree $\tau$. In this case, the representation $o(\sigma)$ is equal to 
the composition product (plethystic substitution) of $\chi^{o(\tau)}$ with $\chi^{(r)}$. In other words,
\begin{align}
\chi^{o(\sigma)} = \chi^{(r)} \circ \chi^{o(\tau)}.
\end{align}
\end{Theorem}

\begin{proof}
Let $m$ denote the number of vertices in $\tau$. Since $\sigma$ has $n$ vertices and 
it contains $r$ copies of $\tau$, $n=mr$. Observe that the wreath product $S_r \wr S_m$ (the 
normalizer of $S_m \times \cdots \times S_m$ ($r$-copies) in $S_n$) acts on $o(\sigma)$
as follows: $S_r$ acts on $S_r \wr S_m$ by permuting the connected blocks (copies of $\tau$)
and $S_m$ acts on individual labeled rooted trees. Let us denote this representation of $S_r \wr S_m$ by $W$. 
Our $S_n$-representation is nothing but the induction of $W$ from $S_r \wr S_m$ to $S_n$. 
It is well-known that the character of such a representation is given by the plethysm
of the corresponding characters. (See Macdonald~\cite{Mac} Appendix A.)
Since the character of permutation representation of $S_r$ on $r$ letters is $\chi^{(r)}$, the proof is complete.

\end{proof}

\begin{Corollary}\label{C:character of a tree representation}
Let $\lambda=(\lambda_1,\dots, \lambda_r)$ be a partition of $n$, and let $\sigma$ be a labeled rooted forest 
comprised of $\lambda_1$ copies of labeled rooted tree $\sigma_1$, $\lambda_2$ copies of labeled rooted tree 
$\sigma_2$, and so on. If the oduns (the underlying trees) of $\sigma_i$'s ($i=1,2,\dots$) are all different from each other, 
then character of $S_n$-module $o(\sigma)$ is given by 
\begin{align}
\chi^{o(\sigma)} = (\chi^{(\lambda_1)} \circ \chi^{o(\sigma_1)}) \cdot (\chi^{(\lambda_2)} \circ \chi^{o(\sigma_2)})\cdot \cdots \cdot
(\chi^{(\lambda_r)} \circ \chi^{o(\sigma_r)}).
\end{align}
\end{Corollary}

\begin{proof}
The only point that we have to be careful is when $m_i =1$ for some $i$. But $\chi^{\md{1}} \circ \chi^\lambda= \chi^\lambda$
for any character $\chi^\lambda$. The rest of the proof follows from Theorem~\ref{T:master plethysm theorem} and Remark~\ref{R:more general}.
\end{proof}



Given $n$, a nonnegative integer, a sequence of positive numbers $(\gamma_1,\dots, \gamma_r)$ is called a 
composition of $n$ if $\sum_{i=1}^r \gamma_i = n$. In the next corollary, the symbol $\models$ stands for the ``composition of''.

\begin{Corollary}
The character of the $k$-file placements $\mc{C}_{k,n}$ is given by 
\begin{align}
\chi^{k,n} 
= \sum_{(m_1,\dots, m_r)\models n -k}\sum_{(o_1,\dots,o_r)} (\chi^{(m_1)} \circ \chi^{o_1}) \cdot (\chi^{(m_2)} \circ \chi^{o_2})\cdot \cdots \cdot (\chi^{(m_r)} \circ \chi^{o_r}),
\end{align}
where the second summation is over all tuples of distinct rooted trees such that $\sum m_i | o_i | = n$.
\end{Corollary}

\begin{proof}
Follows from Corollary~\ref{C:character of a tree representation}.
\end{proof}

\section{The sign representation}\label{S:Sign}

In this section we compute the multiplicity of the sign representation in a forest representation.  
Surprisingly, asymmetry in the underlying rooted tree is a source of regularity in the associated representation.

Let $o=o(\sigma)$ be the odun of a labeled rooted tree and let $F_o$ denote the corresponding Frobenius characteristic. 
Removing the root from $o$ gives a forest $\tau$. 
Let $o_1,\dots, o_r$ be the list of distinct connected components (subtrees) of $\tau$.
By Corollary~\ref{C:character of a tree representation} we see that 
\begin{align}\label{A:Fo}
F_o = s_{(1)} \cdot s_{(i_1)}\circ F_{o_1} \cdot  s_{(i_2)} \circ F_{o_2} \cdot \cdots \cdot s_{(i_r)} \circ F_{o_r},
\end{align}
where $i_1,\dots, i_r$ are the multiplicities of the subtrees $o_1,\dots, o_r$ in the listed order. 
Let us denote by $F_\tau$, the symmetric function $h=F_o/s_{(1)}=\prod_{j=1}^r s_{(i_j)}\circ F_{o_j}$.
(Thus, $F_\tau$ is the Frobenius characteristic of the odun of the forest that is obtained from $o$ by removing its root.)
Determining the full decomposition of $F_o$ into Schur polynomials seems to be difficult because of plethysms.
In this section, we are going to focus on computing the coefficient of $s_{(1^m)}$ in $F_o$ only. 
Our basic observation is that $s_{(1^m)}$ can occur in $F_o$ only if $s_{(1^{k_i})}$ occurs 
in $s_{(i_j)}\circ F_{o_j}$ for $i=j,\dots, r$. Therefore, we are going to focus initially on the multiplicity of 
$s_{(1^m)}$ in $s_{(k)}\circ F_o$. Let us first compute $\langle s_{(1^m)}, s_\lambda \circ F_o \rangle$ 
for some partition $\lambda$.

We start with a more general formula.
\begin{align}
s_{\lambda}\circ F_o &= s_{\lambda}\circ (p_1 h) \qquad (\mt{since $F_o=p_1h$ for some } h=F_\tau\in \Lambda) \notag \\
&=  \sum_{\mu,\nu} \gamma_{\mu,\nu}^{\lambda} (s_\mu \circ p_1) (s_\nu \circ h) \qquad (\mt{by}~(\ref{A:plethysm formula 2})) \notag \\
&=  \sum_{\mu,\nu} \gamma_{\mu,\nu}^{\lambda} s_\mu (s_\nu \circ h) \qquad (\mt{by the axioms of plethysms}). \label{A:by the axioms}
\end{align}

Recall that  $\gamma^\lambda_{\mu,\nu}= \frac{1}{k!} \langle \chi^\lambda, \chi^\mu \chi^\nu \rangle$.
On one hand, whenever $\lambda = (k),\ \mu=(1^k)$ we have 
\begin{align*}
\gamma_{(1^k),\nu}^{(k)} &=\frac{1}{k!} \sum_{w\in S_k}  \chi^{(k)}(w) \chi^{(1^k)}(w) \chi^\nu(w) \\
&=\frac{1}{k!} \sum_{w\in S_k} \chi^{(1^k)}(w) \chi^\nu(w) \qquad (\mt{since $\chi^{(k)}(w) = 1$ for all $w\in S_k$})\\
&=\langle \chi^{(1^k)} , \chi^\nu \rangle \\
&=
\begin{cases}
0 & \mt{if } \nu \neq (1^k),\\
1 & \mt{if } \nu = (1^k).
\end{cases}
\end{align*}
By the same token, but more generally we have 
\begin{align}\label{A:token}
s_{(k)} \circ F_o = \sum_{\mu}  s_\mu (s_\mu \circ h).
\end{align}
Thus, in~(\ref{A:token}) the sign representation $s_{(1^m)}$ occurs if and only if $\mu=(1^k)$ and $s_{(1^{m-k})}$ 
is a summand of $s_\nu \circ h$. Therefore, 
\begin{align}\label{A:an observation}
\langle s_{\lambda}\circ F_o, s_{(1^m)} \rangle = \langle s_{(1^{m-k})}, \sum_{\nu} \gamma_{(1^k),\nu}^{\lambda} s_\nu \circ h \rangle.
\end{align}

On the other hand, $\lambda = (1^k)$ implies that 
\begin{align*}
\gamma_{(1^k),\nu}^{(1^k)} &=\frac{1}{k!} \sum_{w\in S_k}  \chi^{(1^k)}(w) \chi^{(1^k)}(w) \chi^\nu(w) \\
&=\frac{1}{k!} \sum_{w\in S_k}  \chi^\nu(w) \qquad (\mt{since $\chi^{(1^k)}(w) = \pm1$ for all $w\in S_k$})\\
&=
\begin{cases}
0 & \mt{if } \nu \neq (k),\\
1 & \mt{if } \nu = (k).
\end{cases}
\end{align*}
More generally, by using the same idea we obtain $s_{(1^k)} \circ F_o = \sum_{\mu \vdash k}  s_{\mu'} (s_\mu \circ h)$.
The reasoning which we used right after equation~(\ref{A:by the axioms}) gives more: 
\begin{align*}
\langle s_{(1^m)}, s_{\lambda}\circ F_o \rangle &= \langle s_{(1^m)}, \sum_{\mu,\nu} \gamma_{\mu,\nu}^{\lambda} s_\mu (s_\nu \circ h) \rangle \\
&= \langle s_{(1^{m})}, \sum_{\nu} \gamma_{(1^k),\nu}^{\lambda} s_{(1^k)} (s_\nu \circ h) \rangle \qquad \mt{($k$ is equal to $|\lambda|$})\\
&= \langle s_{(1^{m})}, \sum_{\nu} \gamma_{\lambda,\nu}^{(1^k)} s_{(1^k)} (s_\nu \circ h) \rangle \\
&= \langle s_{(1^{m})}, s_{(1^k)} (s_{\lambda'} \circ h) \rangle \\
&= \langle s_{(1^{m-k})}, s_{\lambda'} \circ h \rangle \\
\end{align*}

In conclusion, we have the following `simplification/duality' result:
\begin{Lemma}\label{L:step 1}
Let $F_o$ denote the Frobenius characteristic of an odun $o$ of a rooted tree on $n$ vertices, and let $h=F_\tau$ denote 
the Frobenius characteristic of the odun of the rooted forest on $n-1$ vertices obtained from $o$ by removing its root. In this case, we have
\begin{enumerate}
\item $s_{(k)} \circ F_o = \sum_{\mu \vdash k}  s_\mu (s_\mu \circ h)$,
\item $s_{(1^k)} \circ F_o = \sum_{\mu \vdash k}  s_{\mu'} (s_\mu \circ h)$,
\item $\langle s_{(1^m)}, s_{\lambda}\circ F_o \rangle = \langle s_{(1^{m-k})}, s_{\lambda'} \circ h \rangle$.
\end{enumerate}
In particular, the following equations hold true
\begin{enumerate}
\item[3.1] $\langle s_{(1^m)} , s_{(k)} \circ F_o \rangle = \langle s_{(1^m)}, s_{(1^k)} (s_{(1^{k})}\circ h) \rangle
= \langle s_{(1^{m-k})}, s_{(1^{k})}\circ F_\tau \rangle$,
\item[3.2] $ \langle s_{(1^m)} , s_{(1^k)} \circ F_o \rangle = \langle s_{(1^m)}, s_{(1^k)} (s_{(k)}\circ h) \rangle
= \langle s_{(1^{m-k})}, s_{(k)}\circ F_\tau \rangle$.
\end{enumerate}
\end{Lemma}

\vspace{1cm}

As far as the sign representation is concerned, we have the following crucial definition.
\begin{Definition}\label{D:blossoming}
A (rooted) subtree $a$ of a rooted tree is called a {\em terminal branch} (TB for short) if any vertex of $a$
has at most 1 successor. A {\em maximal terminal branch} (or, MTB for short) is a terminal branch that is not a subtree
of any terminal branch other than itself. 
The length (or height) of a TB is the number of vertices it has. 
We call a rooted tree {\em blossoming} if all of its MTB's are of even length, 
or no two odd length MTB's of the same length are connected to the same parent. 
A rooted tree which is not blossoming is called {\em dry}. 
\end{Definition}

\begin{Lemma}\label{L:step 2}
If an odun $o$ is an MTB of length $l$, then for any $m\geq 2$
$$
\langle s_{(m)}\circ F_o, s_{(1^{ml})} \rangle = 
\begin{cases}
1 & \mt{ if $l$ is an even number};\\
0 & \mt{ otherwise.} 
\end{cases}
$$
\end{Lemma}

\begin{proof}
Observe that removing the root from an MTB results in another MTB whose length is one less than the original's. 
The rest of the proof follows from Lemma~\ref{L:step 1} and axiomatic properties of plethysm. 
\end{proof}

We proceed with extending our Definition~\ref{D:blossoming} to forests. 
\begin{Definition}
We call a rooted forest blossoming if the rooted tree obtained by adding a new root 
to the forest is a blossoming tree. Otherwise, the forest is called dry. 
\end{Definition}
Note that if a forest is blossoming, then all of its connected components are blossoming. 
Also, if a single connected component is dry, then the whole forest is dry. 
In Figure~\ref{F:blossoming} we have listed all blossoming forests 
up to $4$ vertices.

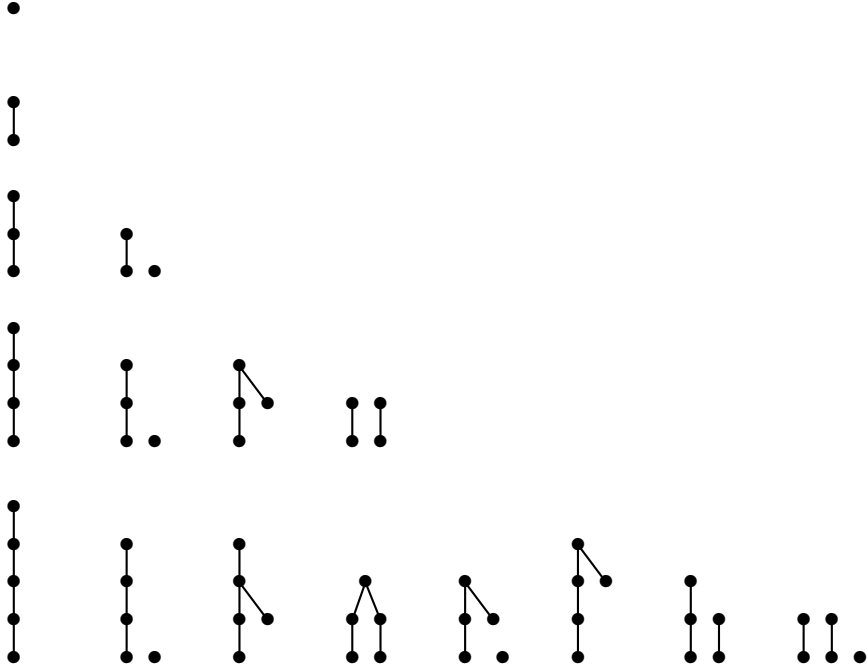
\begin{figure}[htp]
\centering
\begin{tikzpicture}[scale=.5]

\begin{scope}[yshift=3.5cm]
\node at (0,0) {$\bullet$};
\end{scope}

\begin{scope}[yshift=0cm]
\node at (0,0) {$\bullet$};
\node at (0,1) {$\bullet$};
\draw[-, thick] (0,0) to  (0,1);
\end{scope}

\begin{scope}[yshift=-3.5cm]
\node at (0,0) {$\bullet$};
\node at (0,1) {$\bullet$};
\node at (0,2) {$\bullet$};
\node at (3,0) {$\bullet$};
\node at (3,1) {$\bullet$};
\node at (3.75,0) {$\bullet$};
\draw[-, thick] (0,0) to  (0,2);
\draw[-, thick] (3,0) to  (3,1);
\end{scope}

\begin{scope}[yshift=-8cm]
\node at (0,0) {$\bullet$};
\node at (0,1) {$\bullet$};
\node at (0,2) {$\bullet$};
\node at (0,3) {$\bullet$};
\node at (3,0) {$\bullet$};
\node at (3,1) {$\bullet$};
\node at (3,2) {$\bullet$};
\node at (3.75,0) {$\bullet$};
\node at (6,0) {$\bullet$};
\node at (6,1) {$\bullet$};
\node at (6,2) {$\bullet$};
\node at (6.75,1) {$\bullet$};
\node at (9,0) {$\bullet$};
\node at (9,1) {$\bullet$};
\node at (9.75,0) {$\bullet$};
\node at (9.75,1) {$\bullet$};
\draw[-, thick] (0,0) to  (0,3);
\draw[-, thick] (3,0) to  (3,2);
\draw[-, thick] (6,0) to  (6,2);
\draw[-, thick] (6,2) to  (6.75,1);
\draw[-, thick] (9,0) to  (9,1);
\draw[-, thick] (9.75,0) to  (9.75,1);
\end{scope}

\begin{scope}[yshift=-13.75cm]
\node at (0,0) {$\bullet$};
\node at (0,1) {$\bullet$};
\node at (0,2) {$\bullet$};
\node at (0,3) {$\bullet$};
\node at (0,4) {$\bullet$};

\node at (3,0) {$\bullet$};
\node at (3,1) {$\bullet$};
\node at (3,2) {$\bullet$};
\node at (3,3) {$\bullet$};
\node at (3.75,0) {$\bullet$};

\node at (6,0) {$\bullet$};
\node at (6,1) {$\bullet$};
\node at (6,2) {$\bullet$};
\node at (6,3) {$\bullet$};
\node at (6.75,1) {$\bullet$};

\node at (9,0) {$\bullet$};
\node at (9,1) {$\bullet$};
\node at (9.75,0) {$\bullet$};
\node at (9.75,1) {$\bullet$};
\node at (9.35,2) {$\bullet$};

\node at (12,0) {$\bullet$};
\node at (12,1) {$\bullet$};
\node at (12,2) {$\bullet$};
\node at (13,0) {$\bullet$};
\node at (12.75,1) {$\bullet$};

\node at (15,0) {$\bullet$};
\node at (15,1) {$\bullet$};
\node at (15,2) {$\bullet$};
\node at (15,3) {$\bullet$};
\node at (15.75,2) {$\bullet$};

\node at (18,0) {$\bullet$};
\node at (18,1) {$\bullet$};
\node at (18,2) {$\bullet$};
\node at (18.75,0) {$\bullet$};
\node at (18.75,1) {$\bullet$};

\node at (21,0) {$\bullet$};
\node at (21,1) {$\bullet$};
\node at (21.75,0) {$\bullet$};
\node at (21.75,1) {$\bullet$};
\node at (22.5,0) {$\bullet$};

\draw[-, thick] (0,0) to  (0,4);
\draw[-, thick] (3,0) to  (3,3);
\draw[-, thick] (6,0) to  (6,3);
\draw[-, thick] (6,2) to  (6.75,1);
\draw[-, thick] (9,0) to  (9,1);
\draw[-, thick] (9.75,0) to  (9.75,1);
\draw[-, thick] (9.75,1) to  (9.35,2);
\draw[-, thick] (9,1) to  (9.35,2);
\draw[-, thick] (12,0) to  (12,2);
\draw[-, thick] (12,2) to  (12.75,1);
\draw[-, thick] (15,0) to  (15,3);
\draw[-, thick] (15,3) to  (15.75,2);
\draw[-, thick] (18,0) to (18,2);
\draw[-, thick] (18.75,0) to (18.75,1);
\draw[-, thick] (21,0) to (21,1);
\draw[-, thick] (21.75,0) to (21.75,1);

\end{scope}

\end{tikzpicture}
\caption{Blossoming forests on $1\leq n \leq 5$ vertices.}
\label{F:blossoming}
\end{figure}

\begin{Proposition}\label{P:step 3}
Let $\lambda$ be a partition with $|\lambda|=l > 1$ and let $\tau$ be a rooted forest on $m$ vertices.
\begin{enumerate}
\item 
If $\tau$ is a blossoming forest, then 
$
\langle s_{\lambda}\circ F_\tau, s_{(1^{ml})} \rangle = 
\begin{cases}
1 & \mt{ if $\lambda = (1^l)$};\\
0 & \mt{ otherwise.} 
\end{cases}
$
\item If $\tau$ is a dry forest, then $\langle s_{\lambda}\circ F_\tau, s_{(1^{ml})} \rangle =0$. 
\end{enumerate}
\end{Proposition}

\begin{proof}

We prove both of our claims by induction on $m$. It is straightforward to verify them for $m=1$ and $2$,
so, we assume that our claim is true for all forests with $|\tau| < m$. 
Suppose $o_1,\dots, o_r$ are the oduns of the connected components of the forest $\tau$. Note that 
if $\tau$ is a rooted tree, then by part 3 of Lemma~\ref{L:step 1} our problem reduces to the forest case.
 
Let $g_1$ denote $s_{(i_1)}\circ F_{o_1}$ and let $g_2$ denote 
$\prod_{j=2}^r s_{(i_j)} \circ F_{o_j}$ so that $F_\tau = g_1g_2$. 
\begin{align*}
s_{\lambda}\circ F_\tau &= s_{\lambda}\circ (g_1 g_2) \\
&=  \sum_{\mu,\mu} \gamma_{\mu,\nu}^{\lambda} (s_\mu \circ (s_{(i_1)}\circ F_{o_1})) (s_\nu \circ g_2) \\
&=  \sum_{\mu,\nu} \gamma_{\mu,\nu}^{\lambda} 
\left( \sum_{\tilde{\mu}} \left(\sum_{\tilde{\nu}} z_{\tilde{\nu}}^{-1} \chi^{\mu}_{\tilde{\nu}} 
K_{\tilde{\mu},i_1\tilde{\nu}}^{(\tilde{\nu})}  \right) (s_{\tilde{\mu}} \circ F_{o_1}) \right) (s_{\nu} \circ g_2).
\end{align*}
Now we are ready to start the induction argument. 
On one hand, if $\tau$ is dry, at least one of its rooted subtrees is dry. 
Without loss of generality, let $o_1$ denote the dry one. 
Thus, $\langle s_{\tilde{\mu}} \circ F_{o_1}, s_{(1^{|\tilde{\mu}||o_1|})} \rangle=0$ for all partitions $\tilde{\mu}$, 
implying that $\langle s_{\lambda}\circ F_\tau, s_{(1^{ml})} \rangle = 0$. 
On the other hand, if $\tau$ is blossoming, all of its rooted subtrees are blossoming. Therefore, 
$$
\langle s_{\tilde{\mu}} \circ F_{o_1}, s_{(1^{|\tilde{\mu}||o_1|})} \rangle = 
\begin{cases}
1 & \mt{ if $\tilde{\mu}$ is of the form $(1^s)$, $s=|\tilde{\mu}|$};\\
0 & \mt{ otherwise.} 
\end{cases}
$$
In this case, that is when $\tilde{\mu}=(1^s)$, we calculate that $K_{\tilde{\mu},i_1\tilde{\nu}}^{(\tilde{\nu})} = 1$. 
(Follows from the explicit description of the generalized Kostka numbers as given in~\cite{Doran}).
Therefore, 
$$
\sum_{\tilde{\nu}} z_{\tilde{\nu}}^{-1} \chi^{\mu}_{\tilde{\nu}} K_{\tilde{\mu},i_1\tilde{\nu}}^{(\tilde{\nu})} 
=\sum_{\tilde{\nu}} z_{\tilde{\nu}}^{-1} \chi^{\mu}_{\tilde{\nu}} 
=
\begin{cases}
1 & \text{ if $\mu = (s)$ };\\
0 & \text{ otherwise.}
\end{cases}
$$
But it follows from $\mu = (s)$ that $\gamma_{\mu,\nu}^{\lambda}=\delta_{\lambda,\nu}$, the Kronecker delta function. 
Thus, $\langle \lambda s_{(1^m)}, s_{\lambda}\circ g_1g_2 \rangle$ reduces to $\langle \lambda s_{(1^m)}, s_{\lambda}\circ g_2 \rangle$,
which, by induction, is equal to 1 if $\tau$, hence $g_2$ is blossoming. 

\end{proof}

As an application of Proposition~\ref{P:step 3} we determine the multiplicity of the sign representation in $\mc{C}_{k,n}$. 
It boils down to the counting of blossoming trees.

\begin{Theorem}\label{T:sign}
In $\mc{C}_{n-1,n}$ the sign representation occurs exactly $2^{n-3}$ times,
and in $\mc{C}_n$ the sign representation occurs $2^{n-2}$ times. 
\end{Theorem}

By Corollary~\ref{C:character of n-1}, it suffices to prove that $\langle \chi^{n}, \chi^{(1^n)} \rangle  = 2^{n-2}$. 
We cast our problem in symmetric function language. 
Let $T_{n}$ denote the Frobenius characteristic $\mt{ch}(\chi^{n})$,
and let $T_{n-1,n}$ denote $\mt{ch}(\chi^{n-1,n})$. We already know that $T_{n-1,n} = s_{(1)} T_{n}$, 
hence that $\langle T_{n-1,n}, s_{(1^n)} \rangle  = \langle T_{n-1}, s_{(1^{n-1})} \rangle$. 
Therefore, by Proposition~\ref{P:step 3}, it suffices to find the number of 
blossoming trees on $n$ vertices.

\begin{Proposition}\label{P:blossoming}
The number of blossoming forests on $n$ vertices is $2^{n-2}$. 
\end{Proposition}
\begin{proof}
Let $a_n$ denote the number of blossoming forests on $n$ vertices without any isolated vertices, 
and let $b_n$ denote the number of blossoming forests on $n$ vertices with an isolated vertex. 
Set $d_n= a_n + b_n$. Clearly, $d_n$ is the total number of blossoming forests on $n$ vertices. 
Few values of $a_n$ and $b_n$'s are $a_1=a_2=a_3=1, a_4= 3,a_5=5$, and $b_1=b_2=0,b_3=b_4=1,b_5=3$. 
See Figure~\ref{F:blossoming}.

There are obvious relations among $a_n$'s, $b_n$'s, and $d_n$'s. 
For example, adding an isolated vertex to a blossoming forest without an isolated vertex gives 
\begin{align}\label{A:ab 1}
b_{n+1} = a_n \ \mt{ for all $n\geq 1$.}
\end{align}
Similarly, we obtain all blossoming forests on $n$ vertices with no isolated vertex by attaching a new 
single vertex to the isolated vertex of a forest, or by attaching a new root to all of the connected components. 
We depict this in Figure~\ref{F:add a new vertex}. 
The relation we obtain here is 
\begin{align}\label{A:ab 2}
a_{n} = b_{n-1} + d_{n-1}= 2b_{n-1} + a_{n-1} \ \mt{ for $n\geq 3$.}
\end{align}
By combining (\ref{A:ab 1}) and (\ref{A:ab 2}) we arrive at a single recurrence,
\begin{align}\label{A:ab 3}
a_{n} = 2a_{n-2} + a_{n-1}
\end{align}
with initial conditions $a_0=a_1=0$, $a_2=1$. 
\begin{figure}[htp]
\centering
\begin{tikzpicture}[scale=.5]

\begin{scope}
\node at (0,0) {$\bullet$};
\node at (0,1) {$\bullet$};
\node at (0,2) {$\bullet$};
\node at (1,0) {$\bullet$};
\node at (0.75,1) {$\bullet$};
\draw[-, thick] (0,0) to  (0,2);
\draw[-, thick] (0,2) to  (0.75,1);
\end{scope}

\begin{scope}[xshift=5.5cm,yshift=4.5cm]
\node at (0,0) {$\bullet$};
\node at (0,1) {$\bullet$};
\node at (0,2) {$\bullet$};
\node at (2,1) {$\bullet$};
\node at (2,0) {$\bullet$};
\node at (0.75,1) {$\bullet$};
\draw[-, thick] (0,0) to  (0,2);
\draw[-, thick] (2,0) to  (2,1);
\draw[-, thick] (0,2) to  (0.75,1);
\end{scope}

\begin{scope}[xshift=5.5cm,yshift=-4.5cm]
\node at (0,0) {$\bullet$};
\node at (0,1) {$\bullet$};
\node at (0,2) {$\bullet$};
\node at (2,2) {$\bullet$};
\node at (1,3) {$\bullet$};
\node at (0.75,1) {$\bullet$};
\draw[-, thick] (0,0) to  (0,2);
\draw[-, thick] (0,2) to  (1,3);
\draw[-, thick] (2,2) to  (1,3);
\draw[-, thick] (0,2) to  (0.75,1);
\end{scope}

 \draw [->, thick,dashed] (2,2) -- (4,4);
  \draw [->, thick,dashed] (2,0) -- (4,-2);

\end{tikzpicture}
\caption{Adding a new root to blossoming forests.}
\label{F:add a new vertex}
\end{figure}
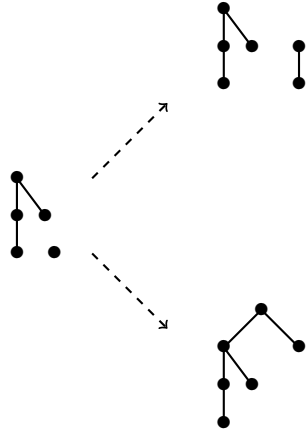
Let $f(x)$ denote the generating series $f(x) = \sum_{n\geq 0 } a_n x^n$.
A straightforward generating function computation gives
\begin{align*}
f(x)  = \frac{x^2}{1-x-x^2}.
\end{align*}
Denoting the generating series of $d_n$ by $G(x)$, the relationship $d_n = a_n + b_n = a_n + a_{n-1}$ 
tells us
\begin{align}
G(x) = f(x) + x f(x) = \frac{x^2+ x^3}{1-x-x^2} = \frac{x^2}{1-2x}
\end{align}
whose power series expansion is $G(x)= \sum_{n\geq 2} 2^{n-2} x^n$.

\end{proof}

\begin{proof}[Proof of Theorem~\ref{T:sign}]
It is immediate from Proposition~\ref{P:blossoming}.
\end{proof}

\begin{Remark}

Let $T_{n}$ denote the Frobenius characteristic $\mt{ch}(\chi^{n-1,n})$, and let $Y=Y(x)$ defined by 
$$
Y= \sum_{n=0}^\infty T_n \frac{x^n}{n!} \qquad (T_0=1)
$$
denote its generating function. We think of generating series $H = \sum_{n= 0}^\infty s_{(n)}$ 
as an operator acting (on the left) on series of symmetric functions by plethysm.
By~(\ref{A:Fo}) we see that the Frobenius characteristic series of $\chi^{k,n}$ is $\frac{1}{k!}s_{(1)}(H\circ Y)^k$. 
Therefore, a l\'a Polya, the functional equation that is satisfied by $Y$ is found to be 
$$
Y = s_{(1)}x + s_{(1)}xH\circ Y + \frac{s_{(1)}x (H\circ Y)^2}{2!} +  \frac{s_{(1)}x (H\circ Y)^3}{3!} + \cdots   = xs_{(1)} e^{H\circ Y}.
$$
Equivalently, we have 
\begin{align}\label{A:Lagrange}
\frac{Y(x)}{\mc{H}(Y(x))}= x,
\end{align}
where $\mc{H}$ is the operation of applying the operator $H$ (plethysitically) on the left, exponentiating the result,  
and then multiplying it by $s_{(1)}$. 
Thus, the solution to the Lagrange inversion problem~(\ref{A:Lagrange}) gives us the Frobenius characteristic of $\chi^{n-1,n}$.
\end{Remark}

\section{Non-attacking nilpotent rooks}\label{S:Rooks}

Non-attacking  rook placements are special file placements, hence they correspond to special labeled rooted forests. 
Indeed, the odun of a non-attacking nilpotent $k$-rook placement has exactly $n-k$ connected components, each of which 
is a rooted tree whose vertices have at most one sibling. In other words, each connected component is a chain. 
Since the ordering of these components does not play a role as far as the underlying labeled structure is concerned, 
the odun of the forest is completely determined by the sizes of the corresponding chains. 
Therefore, we have a one-to-one correspondence between oduns of nilpotent non-attacking $k$-rook placements on an $n\times n$ board
and the partitions of $n$ with exactly $n-k$ parts. 

Let us denote the Frobenius characteristic of the nilpotent non-attacking $n-k$-rook placements by $Z_{k,n}$. 
By the discussion above and by (\ref{A:Fo}) we re-express $Z_{k,n}$ as in 
\begin{align}
Z_{k,n} = \sum_{(\lambda_1,\dots, \lambda_k)=(1^{m_1} 2^{m_2} \dots) \vdash n} s_{(m_i)} \circ p_1^{i},
\end{align}
where $m_i$ is the number of times the part $i$ occurs in $\lambda$. 
Let us denote by $\mc{R}_{k,n}$ the set of all non-attacking $k$-rook placements on $[n]\times [n]$. 
Since $\mc{R}_n = \bigcup_{k=1}^n \mc{R}_{k,n}$, the corresponding Frobenius characteristic 
of nilpotent rook placements is $Z_n:=\sum_{k=1} Z_{k,n}$.

\begin{Theorem}
The number of nilpotent non-attacking $n-k$-rook placements is given by 
$$
\langle Z_{k,n}, s_{(1)}^n \rangle = \langle s_{(1)}^n, \sum_{(\lambda_1,\dots, \lambda_k)=(1^{m_1} 2^{m_2} \dots) \vdash n} s_{(m_i)} \circ p_1^{i}
\rangle.
$$
\end{Theorem}
\begin{proof}
If $\mc{F}_V \in \Lambda^{n}$ is the Frobenius characteristic of an $S_n$-module $V$, then 
$\dim V  = \langle \mc{F}_V, p_1^n \rangle$.
\end{proof}

\begin{Theorem}
For any $n\geq 2$ and $1\leq k \leq n-1$, the total number of irreducible representations in $\mt{Nil}(\mc{R}_{n-k,n})$ is equal to
$$
p(k,n)=\# \mt{number of partitions of $n$ with $k$ non-zero parts}.
$$
Similarly, the occurrence of the sign representation in $\mt{Nil}(\mc{R}_{n-k,n})$ is equal to the number of partitions of $n$ 
with $k$ parts $(\lambda_1,\dots,\lambda_k)$ such that for $i=1,\dots, k$, $\lambda_i$ is even, or $\lambda_j \neq \lambda_i$
for all $j\in [n]-\{i\}$.
\end{Theorem}
\begin{proof}
The proof follows from the discussion above and Proposition~\ref{P:step 3}.
\end{proof}

\section{Final remarks}\label{S:Final}

Suppose $\lambda = (\lambda_1,\dots, \lambda_k)=(1^{m_1} 2^{m_2} \dots) \vdash n$ is the partition type of 
the odun $o(\tau)$ of a nilpotent $n-k$ non-attacking rook placement $\tau$. The scalar product 
$\langle F_{o(\tau)}, s_{(1)}^n \rangle$ gives the dimension of the corresponding $S_n$-module. 
Since we are working with permutation representations, the cardinality of an $S_n$-set gives the dimension of 
the corresponding representation, therefore, the number of labelings of the odun $o(\tau)$ is equal to the dimension 
of the corresponding forest representation. Now, the formula 
\begin{align}\label{A:dimension}
\dim \mc{O}_{o(\tau)} = \langle F_{o(\tau)}, s_{(1)}^n \rangle = \frac{n !} {m_1 ! m_2 ! \cdots }
\end{align}
is easily verified.


A rooted tree $o$ is a poset with unique maximal element and its Hasse diagram contains no cycles. 
If $a \in o$ is a vertex, then its {\em hook} is defined to be $H_a:=\{b \in o:\ b \leq a \}$.
The corresponding {\em hook-length} is $h(a):=|H_a|$. 
A {\em natural labeling} on $o$ is a bijection $g:o \rightarrow [n]$ such that $a < b$ implies $g(a) > g(b)$. 
The famous `hook-length formula' of Knuth~\cite{Knuth3} which 
is proven by Sagan in his thesis~\cite{SaganThesis} asserts that the number of natural labelings of $o$ is equal to 
\begin{align}\label{A:KnuthSagan}
f^\sigma = \frac{n!}{\prod_{a\in o} h(a)}.
\end{align}
Let $\tau$ be as in the previous paragraph so that its odun $o(\tau)$ consists of $m_1$ chains of length 1, 
$m_2$ chains of length 2, and so on.
We add a new root to $\tau$ to obtain a rooted tree $\sigma= \sigma_\tau$. 
Since $\langle F_o , p_1^{n+1} \rangle = \langle F_\tau, p_1^n \rangle$, by~(\ref{A:dimension})
$\dim F_o =  \frac{n !} {m_1 ! m_2 ! \cdots }$. 
On the other hand, by~(\ref{A:KnuthSagan}), we see that $f^\sigma = \frac{(n+1)!}{(2!)^{m_1} (3!)^{m_2}\cdots}$, 
which is different than $\dim F_o$. In the next subsection we explain a more general dimension formula for the dimension 
of a rooted forest representation.

\subsection{Dimension of a forest representation}

In the remainder of this subsection $\sigma$ denotes an unlabeled rooted forest.
If $a$ is a vertex of $\sigma$, then we denote by $\sigma_a^0$ the rooted subforest 
$\{ b \in \sigma :\ b < a \}$, and denote by $\sigma_a$ the rooted subtree $\{ b \in \sigma :\ b \leq a \}$.
Finally, let $\gamma_a$ denote the set $\gamma_a=\{a_1,\dots, a_r\}$, the complete list of children of $a$ whose corresponding 
subtrees $\sigma_{a_i}$ are distinct, that is to say $\sigma_{a_i} \neq \sigma_{a_j}$ if $1\leq i \neq j \leq r$. 
In this case, we denote by $m(a;a_i)$ ($i=1,\dots, r$) the multiplicity of $\sigma_{a_i}$ in $\sigma_a^0$. 

\begin{Theorem}

If $\sigma$ is either a rooted tree on $n+1$ vertices, or a rooted forest on $n$ vertices,
then the dimension of the corresponding representation is 
\begin{align}\label{A:ms}
\frac{ n! } { \prod_{a\in \sigma} \prod_{b \in \gamma_a} m(a;b)!}.
\end{align}

\end{Theorem}

\begin{proof}
We start with the assumption that $\sigma$ is a rooted tree on $n+1$ vertices. Let $\sigma^0$ denote the forest obtained 
from $\sigma$ by removing the root. Since the dimension of the representation is equal to $\langle F_\sigma , s_{(1)}^{n+1} \rangle
=\langle F_{\sigma^0} , s_{(1)}^{n} \rangle$, it suffices to prove our claim for rooted forests on $n$ vertices. 

Towards this end we choose arbitrarily a labeled rooted forest $\sigma$. 
For each vertex $a$ of $\sigma$ and a subtree $\sigma_{b}$, where $b\in \gamma_a$, the permutation of 
$m(a;b)$ copies of $\sigma_b$ does not change the labeled forest $\sigma$. Our claim follows from this observation.
(See Figure~\ref{F:different but the same} for a simple example.) 
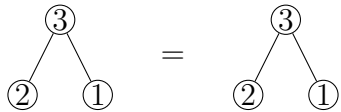
\begin{figure}[htp]
\centering
\begin{tikzpicture}
\begin{scope}[xshift=-1cm]
\node[inner sep =0.5, circle,draw] (1) at (0,0) {1};
\node[inner sep =0.5, circle,draw] (2) at (-1,0) {2};
\node[inner sep =0.5, circle,draw] (3) at (-.5,1) {3};
\draw[-] (3) to (2);
\draw[-] (3) to (1);
\end{scope}
\node at (0,0.5) {$=$};
\begin{scope}[xshift=1cm]
\node[inner sep =0.5, circle,draw] (1) at (0,0) {2};
\node[inner sep =0.5, circle,draw] (2) at (1,0) {1};
\node[inner sep =0.5, circle,draw] (3) at (.5,1) {3};
\draw[-] (3) to (2);
\draw[-] (3) to (1);
\end{scope}
\end{tikzpicture}
\caption{Permuting identical neighboring subtrees does not change the labeling.}
\label{F:different but the same}
\end{figure}

\end{proof}

\section{Decomposition tables}\label{S:Tables}

\subsection{$n=4$}

\begin{align*}
\mc{C}_{0,4} &=V_{(4)} \\ 
\mc{C}_{1,4} &=V_{(4)} \oplus V_{(3,1)}^2 \oplus V_{(2,2)} \oplus V_{(2,1,1)}\\
\mc{C}_{2,4} &= V_{(4)}^3 \oplus V_{(3,1)}^6 \oplus V_{(2,2)}^5 \oplus V_{(2,1,1)}^5 \oplus V_{(1,1,1,1)}^2 \\ 
\mc{C}_{3,4} &= V_{(4)}^4 \oplus V_{(3,1)}^9 \oplus V_{(2,2)}^5 \oplus V_{(2,1,1)}^7 \oplus V_{(1,1,1,1)}^2
\end{align*}

\subsection{$n=5$}

\begin{align*}
\mc{C}_{0,5} &=V_{(5)} \\
\mc{C}_{1,5} &=V_{(5)} \oplus V_{(4,1)}^2 \oplus V_{(3,2)} \oplus V_{(3,1,1)}\\ 
\mc{C}_{2,5} &= V_{(5)}^3 \oplus V_{(4,1)}^7 \oplus V_{(3,2)}^8\oplus V_{(3,1,1)}^6 \oplus V_{(2,2,1)}^6\oplus V_{(2,1,1,1)}^3\oplus V_{(1,1,1,1,1)}\\
\mc{C}_{3,5} &= V_{(5)}^6 \oplus V_{(4,1)}^{20} \oplus V_{(3,2)}^{22}\oplus V_{(3,1,1)}^{25} \oplus V_{(2,2,1)}^{19}\oplus V_{(2,1,1,1)}^{14}\oplus V_{(1,1,1,1,1)}^3\\
\mc{C}_{4,5} &= V_{(5)}^9 \oplus V_{(4,1)}^{26} \oplus V_{(3,2)}^{28}\oplus V_{(3,1,1)}^{30} \oplus V_{(2,2,1)}^{24}\oplus V_{(2,1,1,1)}^{17}\oplus V_{(1,1,1,1,1)}^4
\end{align*}

\subsection{$n=6$}

\begin{align*}
\mc{C}_{0,6} &= V_{(6)} \\
\mc{C}_{1,6} &= V_{(6)} \oplus V_{(5,1)}^{2} \oplus V_{(4,2)} \oplus V_{(4,1^2)}\\
\mc{C}_{2,6} &= V_{(6)}^{3} \oplus V_{(5,1)}^{7} \oplus V_{(4,2)}^{9} \oplus V_{(4,1^2)}^{6} \oplus  V_{(3^2)}^{3} \oplus  V_{(3,2,1)}^{7} \oplus  V_{( 3,1^3)}^{3} \oplus   V_{( 2^3)}^{2} \oplus V_{( 2^2,1^2)} \oplus V_{( 2,1^4)}\\
\mc{C}_{3,6} &= V_{(6)}^{7} \oplus V_{(5,1)}^{23} \oplus V_{(4,2)}^{35} \oplus V_{(4,1^2)}^{33} \oplus  V_{(3^2)}^{19} \oplus  V_{(3,2,1)}^{47} \oplus  V_{( 3,1^3)}^{24} \oplus   V_{( 2^3)}^{14} \oplus V_{( 2^2,1^2)}^{21} \oplus V_{( 2,1^4)}^{9} \oplus V_{(1^6)}^2\\
\mc{C}_{4,6} &= V_{(6)}^{16} \oplus V_{(5,1)}^{59} \oplus V_{(4,2)}^{96} \oplus V_{(4,1^2)}^{96} \oplus  V_{(3^2)}^{46} \oplus  V_{(3,2,1)}^{142} \oplus  V_{( 3,1^3)}^{83} \oplus   V_{( 2^3)}^{43} \oplus V_{( 2^2,1^2)}^{68} \oplus V_{( 2,1^4)}^{36} \oplus V_{(1^6)}^6 \\
\mc{C}_{5,6} &= V_{(6)}^{20} \oplus V_{(5,1)}^{75} \oplus V_{(4,2)}^{114} \oplus V_{(4,1^2)}^{117} \oplus  V_{(3^2)}^{59} \oplus  V_{(3,2,1)}^{170} \oplus  V_{( 3,1^3)}^{96} \oplus   V_{( 2^3)}^{49} \oplus V_{( 2^2,1^2)}^{83} \oplus V_{( 2,1^4)}^{42} \oplus V_{(1^6)}^8
\end{align*}

\bibliography{References.bib}
\bibliographystyle{plain}
\end{document}